\documentclass[12pt,oneside]{amsart}

\usepackage{amsmath, amsthm, amsfonts, hyperref, graphicx, ifpdf}
\usepackage{mathrsfs,epsfig}
\usepackage[dvipsnames,usenames]{color}
\hypersetup{
   unicode=false,          
   pdftoolbar=true,        
   pdfmenubar=true,        
   pdffitwindow=false,     
   pdfstartview={},    
   pdftitle={},    
   pdfauthor={},     
   pdfsubject={},   
   pdfcreator={},   
   pdfproducer={}, 
   pdfkeywords={}, 
   pdfnewwindow=true,      
   colorlinks=true,       
   linkcolor=blue,          
   citecolor=blue,        
   filecolor=blue,      
   urlcolor=blue          
}

\usepackage{amssymb}
\usepackage{amsmath}
\usepackage{mathrsfs}
\usepackage{verbatim}
\usepackage{tikz}
\usepackage[retainorgcmds]{IEEEtrantools}

\usepackage[margin=30 mm,heightrounded=true,centering]{geometry}
 \theoremstyle{plain}
 \newtheorem{theorem}{Theorem}[section]
 \newtheorem{lemma}[theorem]{Lemma}
 \newtheorem{question}[theorem]{Question}
 \newtheorem{corollary}[theorem]{Corollary}
 \newtheorem{proposition}[theorem]{Proposition}
 
 \newtheorem{excer}[theorem]{Excercises}
      
 \theoremstyle{definition}
 \newtheorem{definition}[theorem]{Definition}
      
 \theoremstyle{remark}
 \newtheorem{remark}[theorem]{Remark}
 
 \newcommand{\be}{\begin{equation}}
\newcommand{\ene}{\end{equation}}
\newcommand{\br}{\begin{remark}}
\newcommand{\er}{\end{remark}}
\newcommand{\bl}{\begin{lem}}
\newcommand{\el}{\end{lem}}
\newcommand{\bcor}{\begin{cor}}
\newcommand{\ecor}{\end{cor}}
\newcommand{\bpro}{\begin{pro}}
\newcommand{\epro}{\end{pro}}
\newcommand{\ben}{\begin{enumerate}}
\newcommand{\een}{\end{enumerate}}
\newcommand{\bp}{\begin{proof}}
\newcommand{\ep}{\end{proof}}
\newcommand{\bpo}{\begin{pro}}
\newcommand{\epo}{\end{pro}}
\newcommand{\beq}{\begin{equation*}}
\newcommand{\eeq}{\end{equation*}}
\newcommand{\bear}{\begin{eqnarray}}
\newcommand{\eear}{\end{eqnarray}}
\newcommand{\beqar}{\begin{eqnarray*}}
\newcommand{\eeqar}{\end{eqnarray*}}
\newcommand{\bt}{\begin{theorem}}
\newcommand{\et}{\end{theorem}}
\newcommand{\bex}{\begin{excer}}
\newcommand{\eex}{\end{excer}}

\newcommand{\cL}{{\mathcal{L}}}

\usepackage{syntonly}

\DeclareMathOperator{\dist}{dist}

\makeatletter
\def\@setcopyright{} 
\def\serieslogo@{}
\makeatother

\begin{document}
\title{The Tits alternative for visibility spaces}

\author{Ran Ji \& Yunhui Wu}

\address{Academy for Multidisciplinary Studies, Capital Normal University, Beijing 100048, P. R. China}
\email[R.~J.]{\ jiran.jbm@gmail.com}

\address{Tsinghua University \& BIMSA, Beijing 100084, P. R. China}
\email[Y.~W.]{yunhui\_wu@mail.tsinghua.edu.cn}

\begin{abstract}
    Let $\Gamma$ be a finitely generated group acting properly discontinuously by isometries on a visibility CAT(0) space $X$ that satisfies the bounded packing property. We prove that $\Gamma$ satisfies the Tits alternative: it is either almost nilpotent or contains a free nonabelian subgroup of rank $2$. In the former case, it is equivalent to that the cardinality of the limit set of $\Gamma$ in the geometric boundary of $X$ is no greater than $2$. As an application of the Tits alternative, we show that any finitely generated torsion group acting properly discontinuously by isometries on such a space must be a finite group and have a global fixed point.
    \end{abstract}

\maketitle

\section{Introduction} \label{intro}

In \cite{MR286898}, Tits established a foundational result on the structure of finitely generated linear groups, now known as \emph{Tits alternative}. It states that such a group is either almost solvable or contains a nonabelian free subgroup. The Tits alternative has since become a central concept in geometric group theory, with important examples including: hyperbolic groups, mapping class groups of surfaces, outer automorphism group of free groups, and so on.

Substantial progress has been made in recent years in establishing the Tits alternative for various typical classes of groups. For instance, Osajda and Przytycki \cite{MR4522694} proved that the Tits alternative holds for groups acting discretely on CAT($0$) triangle complexes. Bell, Huang, Peng, and Tucker \cite{MR4780496} established a strong form of the alternative for semigroups of endomorphisms of the projective line. In \cite{MR4902158}, Horbez proved the Tits alternative for outer automorphism groups of a board class of free products, including outer automorphism groups of right-angled Artin groups.

When studying group actions on negatively curved spaces, one often expects a strong form of the Tits alternative. For example, if $\widetilde{M}$ is a Cartan-Hadamard manifold with pinched negative curvature, Besson, Courtois and Gallot \cite{MR2825167} showed that any group acting properly discontinuously by isometries on $\widetilde{M}$ is either almost nilpotent (i.e. it contains a nilpotent subgroup of finite index) or contains a nonabelian free subgroup and has uniformly positive algebraic entropy. In \cite{MR4275871}, Breuillard and Fujiwara further generalized this quantitative Tits alternative to Gromov hyperbolic spaces with bounded geometry.

It is natural to ask the following question.
\begin{question}\label{q-ta}
\label{main question}
Let $\Gamma$ be a group acting properly discontinuously by isometries on a proper CAT(0) space. Does $\Gamma$ always satisfy the Tits alternative?
\end{question}
This problem is widely open (see e.g. \cite{Be00}). 

Recall that a CAT($0$) space $X$ is called a \emph{visibility space} if every pair of distinct points in the geometric boundary $X(\infty)$ of $X$ can be joined by a geodesic line. Typical examples include CAT ($-1$) spaces, and more generally Gromov hyperbolic CAT($0$) spaces. If a proper visibility space admits a cocompact isometric group action, then it is Gromov hyperbolic. We remark here that a proper visibility space admitting a finite-volume quotient may not be a Gromov hyperbolic space; one may see \cite[Appendix]{JiWu} for an example. Visibility is an important geometric property that lies between nonpositive and uniformly negative curvature. A metric space is said to satisfy \emph{bounded packing property} if for any $r>\delta>0$, there exists a packing number $P=P(r,\delta)>0$ such that any geodesic ball of radius $r$ can be covered by at most $P$ geodesic balls of radius $\delta$. A typical example is a complete Riemannian manifold whose Ricci curvature is bounded from below. The main purpose of this paper is to give a positive answer to Question \ref{q-ta} when the space is a visibility space that satisfies the bounded packing property. More precisely,
\begin{theorem}
\label{Tits alternative}
Let $\Gamma$ be a finitely generated group acting properly discontinuously by isometries on a visibility CAT(0) space $X$ that satisfies the bounded packing property. Then exactly one of the following holds:
\begin{enumerate}
     \item $\Gamma$ contains a nonabelian free subgroup;
    \item $\Gamma$ is almost nilpotent, that is equivalent to $$|\mathcal{L}(\Gamma)|\leq 2$$ where $\cL(\Gamma)\subset X(\infty)$ is the set of limit points of $\Gamma$. 
\end{enumerate}
\end{theorem}

\begin{remark}\label{re-mt}
\begin{enumerate}
\item[(a)] If $X$ is a visibility manifold that covers a noncompact Riemannian manifold $M$ of finite volume with sectional curvature in $[-1,0]$, and $\Gamma$ is the fundamental group of an end of $M$ which consists of parabolic isometries of $X$ satisfying $$|\mathcal{L}(\Gamma)|=1$$ (see \cite{MR577132}), Part $(2)$ of Theorem \ref{Tits alternative} then implies that $\Gamma$ is almost nilpotent. This result was first obtained in \cite{JiWu}, resolving a long-standing folklore conjecture. In contrast, Theorem \ref{Tits alternative} requires neither that $\Gamma$ is torsion-free nor that $X$ admits a finite-volume quotient; moreover, the space $X$ does not even need to be a Riemannian manifold. Therefore, Theorem \ref{Tits alternative} significantly generalizes \cite[Theorem 1.2]{JiWu}, in which the three assumptions are essential.  
\item[(b)] The bounded packing property (equivalent to \emph{bounded geometry}, see Lemma \ref{packing geometry}) is essential for the validity of Theorem \ref{Tits alternative} (see e.g. \cite[Section 6]{MR1243787}). Given a discrete isometric group action of $\Gamma$ on a proper CAT($0$) space $X$, one can use a warped product construction to endow $X \times \mathbb{R}$ with a proper CAT(-1) metric on which $\Gamma$ acts properly discontinuously and the limit set $\cL(\Gamma)$ of $\Gamma$ is a single point in the geometric boundary of $X \times \mathbb{R}$. Clearly, one can choose $\Gamma$ to be not an almost nilpotent group on a suitable CAT(0) space $X$. However, this does not violate Part $(2)$ of Theorem \ref{Tits alternative}, because the warped product $X \times \mathbb{R}$ does not satisfy the bounded packing property. 
\end{enumerate}
\end{remark}

\noindent  We now outline the main ideas in the proof of Theorem \ref{Tits alternative}. If the limit set $\mathcal{L}(\Gamma)\subset X(\infty)$ of $\Gamma$ satisfies $|\mathcal{L}(\Gamma)|\geq 3$, an essential step is to show that $\Gamma$ cannot consist only of torsion elements, then we will apply a Ping-pong argument to show the existence of a nonabelian free subgroup of $\Gamma$. For the remaining case, i.e. the second part of Theorem \ref{Tits alternative}, the relatively difficult part is to show that if $\Gamma$ is infinite and has a common point in $X(\infty)$, then $\Gamma$ is almost nilpotent. When $X$ is the universal cover of a finite-volume visible Riemannian manifold $M$ with curvature in $[-1,0]$ and $\Gamma$ is the fundamental group of an end of $M$, the almost nilpotency of $\Gamma$ was proved in \cite{JiWu}. They firstly showed the subexponential growth of $\Gamma$, therefore, the amenability of $\Gamma$. Then by applying the work of Adams-Ballmann \cite{MR1645958} to a transverse space $X_\xi$ of $X$ and using the induced cocompact action of $\Gamma$ on $X_\xi$, a global fixed point in $X_\xi$ for $\Gamma$ is obtained. Then the almost nilpotency for $\Gamma$ is followed by establishing a type of classical Margulis Lemma at infinity. To extend the work of \cite{JiWu} to a more general setting as stated in Theorem \ref{Tits alternative}, we first apply the work of \cite{JiWu} to deduce that $\Gamma$ is still amenable. Although the group $\Gamma$ does not have cocompact action on a transverse space again, we consider iterations of transverse spaces; using the works \cite{MR1645958} and \cite{MR4275871} of Breuillard-Fujiwara, we obtain a fixed point for $\Gamma$ in an iterated transverse space. This will imply that $$\inf\limits_{x\in X}\max\limits_{1\leq i\leq k}\dist(\phi_i\circ x, x)=0$$ where $\Gamma=\langle \phi_1,\cdots, \phi_k\rangle$. Then with the help of the remarkable work of Breuillard, Green and Tao \cite{MR3090256} in which they established a generalized Margulis Lemma for metric spaces with the bounded packing property that was conjectured by Gromov, we can conclude that $\Gamma$ is almost nilpotent. \\

In Theorem \ref{Tits alternative}, the group $\Gamma$ is allowed to contain elements of finite order. Recall that a \emph{torsion group} is a group in which every element has finite order. It is rather an algebraic fact that every finitely generated almost nilpotent torsion group is finite (see e.g. \cite{MR2825167}). While verifying the almost nilpotency (and hence the finiteness) of torsion groups is a common first step towards the Tits alternative, the question of whether a torsion group acting on a nonpositively (or negatively) curved space must have a global fixed point is central and of independent interest in geometric group theory. In this context, we recall the following well-known open question (see e.g. \cite{Sw99, Be00, HO22}).
\begin{question}
\label{torsion group}
Let $X$ be a proper CAT(0) space, and let $\Gamma$ be a finitely generated torsion group acting properly discontinuously by isometries on $X$. Does $\Gamma$ have a global fixed point, or equivalently, is $\Gamma$ necessarily finite?
\end{question}

By the second part of Remark \ref{re-mt}, one may always assume that the space $X$ in Question \ref{torsion group} is a proper CAT$(-1)$ space, in particular, a visibility space. Meanwhile, we will show that a torsion group acting by isometries on a proper visibility space will always have a unique limit point, in particular fixed by the whole group, in its geometric boundary if the group is unbounded (see Theorem \ref{singleton}). Applying Theorem \ref{Tits alternative} to finitely generated torsion groups, we obtain an affirmative answer to Question \ref{torsion group} in the case of visibility space satisfying the bounded packing property. More precisely,
\begin{theorem}
\label{global fixed point}
Let $X$ be a visibility CAT(0) space that satisfies the bounded packing property. Then any finitely generated torsion group acting properly discontinuously by isometries on $X$ is finite and has a global fixed point.
\end{theorem}

Question \ref{torsion group} remains open even under additional assumptions, such as $X$ being finite-dimensional or $\Gamma$ being a subgroup of a CAT($0$) group. Progress has been made in several special cases. Sageev \cite{Sa95} studied Question \ref{torsion group} for CAT($0$) cubical complexes. Norin, Osajda and Przytycki \cite{NOP22} resolved the case of two-dimensional complexes under mild hypotheses. Additionally, Izeki and Karlsson \cite{IK24} showed that a finitely generated torsion group of subexponential growth cannot act on a finite-dimensional CAT(0) space without a global fixed point. One may refer to the recent paper \cite{HO22} for a nice survey on this problem. In the negative direction, it is known that both Grigorchuk groups and Burnside groups can act on infinite-dimensional CAT($0$) cubical complexes without global fixed points. Serre \cite{Ser80} showed that every infinitely generated group admits an action on a tree without a global fixed point.

\subsection*{Plan of the paper}
This paper is organized as follows. In Section \ref{Preliminaries}, some preliminaries are provided. In Section \ref{torsion}, we prove that a torsion group acting on a proper visibility space has a unique limit point at infinity, if it exists. In Section \ref{sec-fg}, we prove the first part of Theorem \ref{Tits alternative}. We finish the proofs of Theorem \ref{Tits alternative} and \ref{global fixed point} in the last section.  

\subsection*{Acknowledgement}
The authors are very grateful to Jingyin Huang and Anders Karlsson for their insightful and helpful comments and suggestions that improve the quality of this article. The authors also thank Hiroyasu Izeki and Wenyuan Yang for their interest in this work. Y.~W. is partially supported by NSFC grants 12361141813 and 12425107.

\tableofcontents

\section{Preliminaries} \label{Preliminaries}

In this section, we set up certain notation and review essential concepts from the theory of spaces of nonpositively curved geometry. 

\subsection*{CAT($0$) spaces} \label{CAT} \hfill\\

Let $X$ be a complete geodesic metric space. We say $X$ is a \emph{CAT($0$) space} if for any three points $x,y,z \in X$ we have $$d^2(z,m) \leq \dfrac{1}{2}\left(d^2 \left(z,x \right)+d^2\left(z,y\right)\right)-\dfrac{1}{4}d^2(x,y),$$ where $m$ is the midpoint of a geodesic segment from $x$ to $y$. CAT($0$) spaces naturally generalize Cartan-Hadamard manifolds. 

A metric space $X$ is said to be \emph{proper} if every closed bounded set is compact. In a proper CAT($0$) space, the \emph{(Alexandrov) angle} $\angle_p(\gamma_1,\gamma_2)$ between two unit speed geodesics $\gamma_1, \gamma_2$ emanating from the same point $p$ is defined by
\begin{equation*}
\angle_p(\gamma_1,\gamma_2)=\lim_{t \to 0} 2 \arcsin \dfrac{d(\gamma_1(t),\gamma_2(t))}{2t}.
\end{equation*}
We denote by $\gamma_{p,q}$ the unit speed geodesic from $p$ to $q$, and for $x,y \in X$, we write $\angle_p(x,y)$ for the angle between $\gamma_{p,x}$ and $\gamma_{p,y}$.

We say $X$ has \emph{bounded geometry} if for any $r>\delta>0$, there exists $n_{r,\delta} \in \mathbb{N}^*$ such that any geodesic ball of radius $r$ contains the centers of at most $n_{r,\delta}$ disjoint geodesic balls of radius $\delta$. A complete metric space with bounded geometry is always proper.  A typical example of a CAT($0$) space with bounded geometry is a complete Riemannian manifold with Ricci curvature bounded below.

We say that a metric space satisfies \emph{bounded packing property} if for any $r>\delta>0$, there exists a packing number $P=P(r,\delta)>0$ such that any geodesic ball of radius $r$ can be covered by at most $P$ geodesic balls of radius $\delta$.
\begin{lemma}(\cite[(8)]{MR4381234})
\label{packing geometry}
A metric space $X$ has bounded geometry if and only if it satisfies the bounded packing property.
\end{lemma}
\begin{proof}
Suppose that $X$ satisfies the bounded packing property. Given $0<\delta<r$, there exists $P=P(\delta,r)$ such that every geodesic ball $B_x(r)$ can be covered by at most $P$ geodesic balls of radius $\delta$. It follows that there are at most $P$ disjoint geodesic balls $B_{x_i}(\delta)$ with centers $x_i \in B_x(r)$, since a single geodesic ball of radius $\delta$ cannot cover two distinct such centers.

Conversely, suppose that $X$ has bounded geometry. For $0<\delta<r$, let $n=n_{r,\delta/2}$ be the maximal number of disjoint geodesic balls $B_{x_i}(\delta)$ of radius $\delta/2$ with centers in $B_x(r)$. Then we have $$B_x(r) \subset \bigcup_{i=1}^{n_{r,\delta/2}} B_{x_i}(\delta),$$ otherwise for any $y \in B_x(r) \setminus \bigcup_{i=1}^{n_{r,\delta/2}} B_{x_i}(\delta)$, the ball $B_y(\delta/2)$ would be disjoint from all $B_{x_i}(\delta/2)$, contradicting the definition of $n$.
\end{proof}

The following lemma shows that in proper $CAT(0)$ spaces, the packing number for large balls is controlled by that for small balls. For Riemannian manifolds, this can be interpreted as a volume growth condition.  The converse is true under the additional assumption of geodesic extension property, see \cite[Example 4.3]{MR4381234}.

\begin{lemma}(\cite[Lemma 4.7]{MR4381234})
\label{packing}
Let $X$ be a CAT($0$) space. If every geodesic ball of radius $6$ can be covered by at most $P$ balls of radius $1$, then  every geodesic ball of radius $R$ can be covered by at most $(1+P)^R$  balls of radius $1$.
\end{lemma}

\subsection*{The geometric boundary and the cone topology} \hfill\\

Let $X$ be a CAT($0$) space.  Two geodesic rays $\gamma_1$ and $\gamma_2$ in $X$ are said to be \emph{asymptotic}, denoted by $\gamma_1 \sim \gamma_2$, if there exists a constant $C$ such that for any $t \geq 0$ we have
$$d(\gamma_1(t),\gamma_2(t)) \leq C.$$

Define $X(\infty)$, the \emph{geometric boundary} of $X$, to be
$$X(\infty)= \text{the set of all geodesic rays}/ \sim.$$
We denote $\overline{X}=X \cup X(\infty)$ and $\gamma(\infty)$ as the equivalence class of $\gamma$. 

In the case when $\widetilde{M}$ is a Cartan-Hadamard manifold, the Cartan-Hadamard theorem implies that the exponential map at any $p \in \widetilde{M}$ induces a diffeomorphism between the tangent space $\mathrm{T}_p \widetilde{M}$ and  $\widetilde{M}$. The geometric boundary $\widetilde{M}(\infty)$ can thus be canonically identified with the unit tangent sphere $\mathrm{S}_p \widetilde{M}$.  It is natural to endow $\widetilde{M} \cup \widetilde{M}(\infty)$ with a topology such that $\widetilde{M}(\infty)$ is homeomorphic to the sphere $\mathrm{S}_p \widetilde{M}$

Let $x_0 \in X$ be a reference point of a CAT($0$) space $X$. For $r>0$, since the closed geodesic ball $\overline{B_{x_0}(r)}$ is convex, the projection map $p_r:X \to \overline{B_{x_0}(r)}$ is well-defined. Given $\eta \in X(\infty)$ and $\varepsilon,R>0$, We call
\begin{equation*}
T_{x_0}(\eta,r,\varepsilon)=\{x \in \overline{X}: d(x_0,x)>r, d(p_r(x),\gamma_{x_0,\eta}(r))<\varepsilon\}
\end{equation*}
a truncated cone in $\overline{X}$. Then the set of all open balls in $X$ and all truncated cones in $\overline{X}$ form a basis of a topology on $\overline{X}$, which is called the \emph{cone topology}. The cone topology on $X$ is independent of the choice $x_0$.

If $X$ is proper, the cone topology makes $\overline{X}$ a compactification of $X$. In general $\overline{X}$ is not compact.

\subsection*{The horocycle topology} \hfill\\ 

Let $X$ be a CAT($0$) space and $\xi \in X(\infty)$. Denote by $\gamma_{x, \xi}:[0,\infty] \to \overline{X}$ the geodesic ray parametrized by arc length with $\gamma_{x, \xi}(0)=x$ and $\gamma_{x, \xi}(\infty)=\xi$. The \emph{Busemann function} $B_{{x, \xi}}:X \to \mathbb{R}$ is defined by
\begin{equation*}
B_{{x, \xi}}(y)=\lim_{t \to \infty}\left(d\left(\gamma_{x,\xi}\left(t\right),y\right)-t\right).
\end{equation*}
The level sets of $B_{{x, \xi}}$ are called \emph{horospheres} centered at $\xi$ and the sublevel sets of $B_{\gamma_{x, \xi}}$ are called \emph{horoballs}. If two geodesic rays $\gamma_1 \sim \gamma_2$, then $B_{\gamma_1}-B_{\gamma_2}$ is a constant.

When $X$ is proper, there is a unique topology that makes $\overline{X}$ a compactification of $X$, such that the set of horoballs centered at $\xi$ forms a local basis for every $\xi \in X(\infty)$. This topology is called the \emph{horocycle topology}.

\subsection*{Visibility spaces} \hfill\\ 

We review the visibility axiom for spaces. Standard references are \cite{MR0336648} and \cite{MR1744486}. Some of the results therein are stated only for Riemannian manifolds but carries over to general proper metric spaces. We omit the proofs unless the original arguments make essential use of properties of manifolds that are not valid for metric spaces.

\begin{definition}
    A CAT($0$) space $X$ is said to satisfy the \emph{visibility axiom} if every pair of distinct points $\xi,\eta \in X(\infty)$ can be joined by a geodesic, i.e., there is a geodesic $g:\mathbb{R} \to X$ such that $g(-\infty)=\xi$ and $g(\infty)=\eta$. We say $X$ is \emph{visible} or a  \emph{visibility space}.
    \end{definition}

For proper CAT($0$) spaces, the following lemma is often used as an equivalent definition of visibility spaces.
\begin{lemma} (see e.g. \cite[Proposition 9.32]{MR1744486})
    \label{visibility}
    A proper CAT($0$) space $X$ is a visibility spaces if and only if for any $p \in X$ and $\varepsilon>0$, there exists a constant $R=R(p,\varepsilon)>0$ such that if $g$ is a geodesic (segment, ray) satisfying $d(p,g)\geq R$, then $$\angle_p(g) \leq \varepsilon,$$
    where $\angle_p(g)=\sup_{x,y \in g} \angle_p(x,y)$.
    \end{lemma}

As a consequence,
    \begin{lemma}(see e.g. \cite[Proposition 4.7]{MR0336648})
    \label{small viewing angle}
        Let $X$ be a proper visibility space, and let $\{p_n\} \in X$ be a sequence converging to $\xi \in M(\infty)$ in the cone topology. Then for any neighborhood $U$ of $\xi$ in $\overline{X}$, we have
$$\lim\limits_{n\to \infty}\sup \limits_{x,y \in \overline{X}\setminus U}\angle_{p_n}(x,y)= 0.$$
    \end{lemma}   

For visibility spaces, the horocycle topology is strictly larger than the cone topology. That is, for  any neighborhood $U$ of $\eta \in X(\infty)$ in the cone topology, there exist a horoball centered at $\eta$ that is contained in $U$. One may see  \cite[Proposition 4.8]{MR0336648} or \cite[Proposition II.9.35]{MR1744486} for more details.

\subsection*{Isometries of visibility spaces} \hfill\\

Let $\varphi$ be an isometry of a CAT($0$) space $X$. The \emph{displacement function} $d_\varphi:X \to \mathbb{R}$ associated with $\varphi$ is defined by
\begin{equation*}
d_\varphi(x)=d(x,\varphi(x)).
\end{equation*}
We call
\begin{equation*}
|\varphi|=\inf_{x \in X} d_\varphi(x)
\end{equation*}
the \emph{translation length} of $\varphi$. The collection of nontrivial isometries is divided into three disjoint classes.
\begin{definition}
An isometry $\varphi:X \to X$ is called
\begin{IEEEeqnarray*}{RRL}
(\textrm{a})& & \;\; \textrm{\emph{elliptic} \;if } d_\varphi \textrm{ achieves its minimum on } X \textrm{ and } |\varphi|=0.\\
(\textrm{b})& & \;\; \textrm{\emph{axial} \;if } d_\varphi \textrm{ achieves its minimum on } X \textrm{ and } |\varphi|>0.\\
(\textrm{c})& & \;\; \textrm{\emph{parabolic} \;if } d_\varphi \textrm{ does not achieve the minimum on } X.
\end{IEEEeqnarray*}
\end{definition}

An isometry $\varphi$ of a CAT($0$) space $X$ extends naturally to a homeomorphism of $\overline{X}$, which is still denoted by $\varphi$.  Denote by $\mathrm{Fix}_\varphi$ the set of fixed points of $\varphi$ in $\overline{X}$. For a visibility space $X$, we can characterize an isometry $\varphi$ of $X$ by $\mathrm{Fix}_\varphi$.
\begin{theorem}(see e.g. \cite[Lemma 6.8]{MR823981})
\label{fixed point}
Let $\varphi$ be an isometry of a proper visibility space $X$. Then exactly one of the following occurs,
\begin{IEEEeqnarray*}{RRL}
(\textrm{a})& & \;\; \textrm{$\varphi$ is elliptic, in which case $\varphi$ has a bounded orbit and $\mathrm{Fix}_\varphi \neq \emptyset \subset X$}, \\
(\textrm{b})& & \;\; \textrm{$\varphi$ is axial, in which case $\varphi$ has exactly two fixed points $\eta,\xi$ in $X(\infty)$ and $\varphi$} \\
& & \; \; \textrm{translates a geodesic joining $\eta$ and $\xi$},\\
(\textrm{c})& & \;\; \textrm{$\varphi$ is parabolic, in which case $\varphi$ has exactly one fixed point in $X(\infty)$}.
\end{IEEEeqnarray*}
\end{theorem}

It was proved in \cite{Wu18} that every parabolic isometry of a proper visibility space always has vanishing translation length. In accordance with the case of uniformly negative curvature, a proper visibility space  also exhibits the following contracting-expanding property along the axes of axial isometries.
\begin{theorem} (see \cite[Theorem 2.2]{MR656659})
\label{axial}
Let $X$ be a proper visibility space, and  let $\varphi$ be an axial isometry of $X$ and $\gamma$ an axis of $\varphi$. Given disjoint open neighborhoods $U$ of $\gamma(-\infty)$ and $V$ of $\gamma(\infty)$, there exists $N \in \mathbb{N}$ such that for every $n \geq N$, 
$$\varphi^n(\overline{X} \setminus U) \subset V \text{ and }\varphi^{-n}(\overline{X} \setminus V) \subset U.$$
\end{theorem}

And similarly for parabolic isometries we have
\begin{theorem} (see \cite[Proposition 8.4]{MR0336648})
\label{parabolic}
Let $X$ be a proper visibility space and let $\varphi$ be a parabolic isometry of $X$ with the unique fixed point $\eta \in X(\infty)$. Then given any neighborhood $U$ of $\eta$, there exists  $N \in \mathbb{N}$ such that for every $n \geq N$, 
$$\varphi^n(\overline{X} \setminus U) \subset U \text{ and }\varphi^{-n}(\overline{X} \setminus U) \subset U.$$
\end{theorem}

\subsection*{Groups of isometries of visibility manifolds} \hfill\\

We now investigate isometric group actions on proper visibility spaces.

As remarked in the preceding subsection, every isometry of a proper visibility space $X$ fixes at least one point in the compactification $\overline{X}=X \cup X(\infty)$.  Let $\Gamma$ be a group of isometries of  $X$. For a point $x \in \overline{X}$, we denote by $\Gamma_x$ the \emph{stability subgroup} $$\Gamma_x=\{\varphi \in \Gamma:\ \varphi(x)=x \}.$$

A group $\Gamma_0$ of isometries of $X$ is called a \emph{stability group} if it fixes a point $x \in \overline{X}$, i.e., $$\Gamma_0=(\Gamma_0)_x.$$

As the following results show, the case when the group is almost nilpotent is of particular interest in relating the geometry of a visibility manifold to an isometric group action, and is closely connected to the case that the whole group has a fixed point in $X(\infty)$.

\begin{lemma} (see \cite[Lemma 3.1b]{MR577132})
\label{stability nilpotent}
Let $X$ be a proper visibility space, and let $\Gamma_0$ be an almost nilpotent group acting on $X$ freely and properly discontinuous by isometries. Then $\Gamma_0$ has a global fixed point in $X(\infty)$. In particular the elements of $\Gamma_0$ are either all axial or all parabolic.
\end{lemma}

\begin{lemma} (see \cite[Proposition 6.8, Proposition 6.9, Theorem 6.11]{MR0336648})
\label{stability}
Let $\Gamma$ be a group of isometries acting properly discontinuously on a proper visibility space $X$. If $\Gamma$ contains an axial element $\varphi$ which translate an axis with endpoints $\eta,\xi \in X(\infty)$, then every element of the stability group $\Gamma_\eta$ fixes $\xi$ as well. Moreover, $\Gamma_\eta$ is almost infinite cyclic.
\end{lemma}
\begin{proof}
The first assertion is proved in \cite[Proposition 6.8]{MR0336648}. As a consequence, every $\psi \in \Gamma_\eta$ preserves the set of geodesics with endpoints $\eta$ and $\xi$. In particular, an elliptic element in $\Gamma_\eta$ fixes such a geodesic pointwise.  Since $X$ is a  visibility space, the Hausdorff distance between any two parallel geodesics joining $\eta$ and $\xi$ is uniformly bounded. Combined with the proper discontinuity, it follows  that $\Gamma_\eta$ contains only finitely many elliptic elements. By \cite[Theorem 6.11]{MR0336648}, a finite index subgroup of $\Gamma_\eta$ is infinite cyclic.
\end{proof}

We have the following analogue of the Cartan fixed point theorem. 
\begin{lemma}
\label{stability finite}
\label{Cartan}
Let $X$ be a proper CAT($0$) space, and let $\Gamma_0$ be a group acting on $X$ by isometries. Then $\Gamma$ fixes a point $x \in X$ if and only if it has finite orbit.
\end{lemma}

In view of  Lemma \ref{stability nilpotent}, Lemma \ref{stability} and Lemma \ref{stability finite}, to give a complete characterization of nilpotent group actions  on a visibility space in terms of the limit set, the only remaining case is when the group consists of elliptic and parabolic elements. We show that, under the additional assumption of the bounded packing property, such a group is indeed almost nilpotent if its limit set at infinity has no more than two points.

\subsection*{Transverse spaces} \hfill\\

The study of the transverse space associated with a geometric boundary point of a CAT($0$) space was initiated by Caprace and Monod. We refer to \cite{MR2495801} and \cite{CM13} for more general results and the proofs.

Let $X$ be an unbounded CAT($0$) space and $\xi \in X(\infty)$ be a boundary point. Fix a horosphere $H$ centered at $\xi$. Define $$d_\xi(x,y)=d(\gamma_{x, \xi}, \gamma_{y, \xi}) \text{ for } \ x,y \in H.$$ It is easy to see that $(H,d_\xi)$ is a pseudo-metric space. Its metric completion, denoted by $(X_\xi,d_\xi)$, is called the \emph{transverse space} of $\xi$. Associate with the transverse space there is a natural projection map $\pi_{\eta}:X \to X_{\eta}$.

For example, if $X$ is a complete simply connected manifold of uniformly negative sectional curvature, then $d_\xi(x,y)=d_\xi(\pi_\xi(x),\pi_\xi(y))\equiv 0 \text{ for all} \ x,y \in H$. That is, the space $(X_\xi,d_\xi)$ is a single point.

We will use the following property of transverse spaces, which is a reformulation of Proposition 3.3 in \cite{CM13}.
\begin{theorem}
\label{Transverse proper}
The space $X_\xi$ is a complete CAT($0$) space. If $X$ is of bounded geometry, then so is $X_\xi$; in particular, $X_\xi$ is proper.
\end{theorem}

Since the distance function on a CAT(0) space is convex, it is not hard to see that
\begin{equation*}\label{1-lip}
d_\xi(x,y)\leq d(x,y) \text{ for all } x,y \in H.
\end{equation*}

In our proof of Theorem \ref{Tits alternative}, we need a concept associated with the iteration of transverse spaces called the \emph{depth}. The transverse space $M_{\eta_1}$ of $\eta_1 \in M(\infty)$ is a complete CAT($0$) space of bounded geometry. The stability group $\Gamma_0$ fixing $\eta_1$ and all horospheres centered at $\eta_1$ induces a group action on $M_{\eta_1}$. If $M_{\eta_1}$ is unbounded, then for any $\eta_2 \in M_{\eta_1}(\infty)$ we can define the transverse space $M_{\eta_1,\eta_2}$ of $\eta_2$. By iterating this process we get a sequence of proper CAT($0$) spaces $M_{\eta_1}, M_{\eta_1,\eta_2},\cdots$ such that $\eta_{i+1} \in M_{\eta_1,\ldots,\eta_i}(\infty)$. 

\begin{definition}
$\max \{k: \eta_{i+1} \in M_{\eta_1,\ldots,\eta_i}(\infty), \; 1 \leq i \leq k-1\}$ is called the \emph{depth} of $\widetilde{M}$.
\end{definition} 

The following result for the depth of CAT($0$) spaces  is due to Caprace and Monod.
\begin{theorem} (see \cite[Proposition 3.3]{MR3072802})
    \label{finite depth}
    If a CAT($0$) space $X$ is of bounded geometry, then it has finite depth, i.e. the iteration of transverse space terminates after finitely many steps. 
    \end{theorem}

Let $\varphi$ be a parabolic isometry of $X$ fixing $\xi$ and all horospheres centered at $\xi$. Then $\varphi$ induces naturally an (not necessarily parabolic) isometry of $X_\xi$ with the same translation length, which is still denoted by $\varphi$ when it doesn't cause any confusion. It is straightforward to check that if a group of isometries  fixing $\xi$ acts cocompactly on a horosphere centered at $\xi$, then it also acts cocompactly on $X_\xi$.
\\

\section{Infinite torsion groups} \label{torsion} \hfill

Since Theorem \ref{Tits alternative} allows $\Gamma$ to contain elliptic elements, any finitely generated torsion group satisfying the Tits alternative must belong the second category (i.e. be almost nilpotent). Indeed, if it contains a nonabelian free subgroup, then $\varphi \psi$ would have infinite order whenever $\varphi$ and $\psi$ generate a free group, contradicting that every element of a torsion group  has finite order. 

By Lemma \ref{Cartan}, every finite group acting on a CAT($0$) space has a global fixed point. In this section we show that every infinite torsion group acting on a visibility space $X$ must have a global fixed point in $X \cup X(\infty$).

\begin{lemma}
\label{dual}
Let $X$ be a proper visibility space, and let $\Gamma_0$ a torsion group acting on $X$ by isometries. Assume that the orbit of $\Gamma_0$ is unbounded, and let $\eta \in X(\infty)$ be a limit point of $\Gamma_0$. Then the only point dual to $\eta$ with respect to $\Gamma_0$  is $\eta$ itself, i.e., if  $\varphi_n(p) \to \eta$ for some sequence $\{\varphi_n\} \in \Gamma$, then $\varphi^{-1}_n(p) \to  \eta$. 
\end{lemma}
\begin{proof}
Assuming the contrary that we can find a sequence $\{\varphi_n\} \in \Gamma_0$ so that $\varphi_n(p) \to \eta$ and $\varphi^{-1}_n(p) \to \xi \in X(\infty) \setminus \{\eta\}$. Since $X$ is a visibility space, $\eta$ and $\xi$ can be joined by a geodesic $\gamma$. Fix a base point $x_0 \in \gamma$, then $\angle_{x_0}(\eta,\xi)=\pi$. 

By Lemma \ref{visibility}, there exists a constant $R=R(x_0,X)>0$ such that for any $x,y \in X \setminus \{x_0\}$ with $\angle_{x_0}(x,y) \geq \pi/4$, we have
\begin{equation}
\label{visible v3}
d(x,y) \geq d(x_0,x)+d(x_0,y)-R.
\end{equation}

Since $\eta,\xi$ are dual with respect to $\Gamma_0$, there exists $\phi \in \Gamma_0$ such that
\begin{equation*}
d(x_0,\phi(x_0))=d(x_0,\phi^{-1}(x_0)) > R \textrm{ and } \angle_{x_0}(\phi(x_0),\phi^{-1}(x_0)) \geq \dfrac{\pi}{2}. 
\end{equation*}

Let $x_1$ be a fixed point of $\phi$ (which exists since $\phi$ is elliptic). Since $\angle_{x_0}(\phi(x_0),\phi^{-1}(x_0)) \geq \pi/2$,  $$\max \{ \angle_{x_0}(\phi(x_0),x_1) , \angle_{x_0}(\phi^{-1}(x_0),x_1) \} \geq \dfrac{\pi}{4}.$$ Without loss of generality, we may assume that $\angle_{x_0}(\phi(x_0),x_1) \geq \pi/4$. Applying inequality \eqref{visible v3}, we obtain
\begin{equation*}
d(\phi(x_0),x_1) \geq d(x_0,\varphi(x_0))+d(x_0,x_1)-R > d(x_0,x_1).
\end{equation*}
However, since $\phi$ is an elliptic isometry fixing $x_1$, we have
\begin{equation*}
d(\phi(x_0),x_1)=d(\phi(x_0),\phi(x_1))=d(x_0,x_1),
\end{equation*}
which yields a contradiction. Therefore, the only point dual to $\eta$ is $\eta$ itself.
\end{proof}

\begin{remark}
\label{invariant}
Observe that the set of points in $X(\infty)$ dual to $\eta$ is invariant under the action of $\Gamma_0$. To see this, suppose $\xi$ is dual to $\eta$, i.e., there exists a sequence  $\{\phi_n\} \in \Gamma_0$ such that $$\phi_n(p) \to \eta \textrm{ and } \phi^{-1}_n(p) \to \xi.$$ Then for any $\varphi \in \Gamma_0$, we have $$\phi_n \varphi^{-1}(p) \to \eta \textrm{ and } (\phi_n \varphi^{-1})^{-1}(p)=\varphi \phi_n^{-1}(p) \to \varphi(\xi),$$ which shows that $\varphi(\xi)$ is also dual to $\eta$. 
\end{remark}

As pointed out in Remark \ref{invariant}, by Lemma \ref{dual} we know that every limit point of $\Gamma_0$ is fixed by $\Gamma_0$. We now show that $\Gamma_0$ has a unique limit point. For this purpose, we require the following lemma, which generalizes the contracting-expanding property for individual isometries (see Lemma \ref{axial} and Lemma \ref{parabolic})  to isometric groups actions.
\begin{lemma}
\label{North-South}
Let $X$ be a proper visibility space, and let $\Gamma$ be a group acting on $X$ by isometries. Fix $p \in X$. Assume that $\varphi_n(p) \to \eta^+$ and $\varphi_n^{-1}(p) \to \eta^-$ as $n \to \infty$, where $\{\varphi_n\} \in \Gamma$ and $\eta^+,\eta^- \in X(\infty)$ are not necessarily distinct. Then
\begin{equation*}
\varphi_n(x) \to \eta^+ 
\end{equation*}
for any $x \in \overline{X} \setminus  \{\eta^-\}$. Moreover, the convergence is uniform on compact subsets of $\overline{X} \setminus  \{\eta^-\}$.
\end{lemma}
\begin{proof}
It suffices to prove the lemma for points in $X(\infty) \setminus \{\eta^-\}$.  Since $\varphi_n^{-1}(p) \to \eta^-$, it follows from Lemma \ref{small viewing angle} that for any $\xi \in X(\infty) \setminus \eta^-$,
\begin{equation*}
\angle_{p}(\varphi_n(p),\varphi_n(\xi))  =\angle_{\varphi_n^{-1}(p)}(p,\xi) \to 0,
\end{equation*}
as $n \to \infty$. Therefore $\varphi_n(\xi) \to \eta^+$. Given that the visual angle varies continuously, it is standard to conclude that this convergence is uniform. 
\end{proof}

We use the following classical Ping-pong lemma to construct elements of infinite order. It will also be used later to generate free subgroups.
\begin{lemma}[Ping-pong Lemma, see e.g. \cite{Harpe-book}]
\label{Ping-pong Lemma}
Let $G$ be a group acting on a set $Y$, and let $g_1,\cdot,g_n$ be elements of $G$. Suppose that there exist pairwise disjoint nonempty sets $Y_1,\cdots,Y_n \subset Y$ such that
\begin{equation*}
g_i^k(Y_j) \subset Y_i
\end{equation*}
for any $i \neq j$ and any $k \in \mathbb{Z} \setminus \{0\}$. Then the group generated by the $g_1,\cdots,g_n$ is a free group of rank $n$.
\end{lemma}

Now we are ready to prove the main result in this section.

\begin{theorem}
\label{singleton}
Let $X$ be a proper visibility space, and let $\Gamma_0$ be a torsion group acting on $X$ by isometries.  Assume that the orbit of $\Gamma_0$ is unbounded. Then the limit set of a torsion group $\Gamma_0$ in $X(\infty)$ is a singleton $\{\eta\}$. In particular, $\eta$ is fixed by $\Gamma_0$.
\end{theorem}
\begin{proof}
Assume the contrary that there exist sequences  $\{\varphi_i\}, \{\psi_i\} \in \Gamma_0$ such that $\varphi_i(p) \to \eta^+$ and $\psi_i(p) \to \eta^-$, where $\eta^+$ and $\eta^-$ are distinct points in $X(\infty)$. By Lemma \ref{dual}, we have $\varphi_i^{-1}(p) \to \eta^+$ and $\psi_i^{-1}(p) \to \eta^-$. Let $U,V$ be disjoint neighborhoods of $\eta^+$ and $\eta^-$, respectively. By Lemma \ref{North-South}, for sufficiently large $I, J \in \mathbb{N}^*$,
\begin{equation*}
\varphi_I^{\pm}(\overline{X}\setminus U) \subset U \textrm{ and } \psi_J^{\pm}(\overline{X} \setminus V) \subset V,
\end{equation*}
which implies that
\begin{equation*}
\varphi_I^{\pm}(V) \subset U \textrm{ and } \psi_J^{\pm}(U) \subset V.
\end{equation*}

\noindent Now we apply the Ping-pong Lemma, i.e. Lemma \ref{Ping-pong Lemma}, to conclude that $\varphi_I$ and $\varphi_J$ generate a free group. In particular, the element $\varphi_I \varphi_J$ has infinite order, contradicting our assumption that $\Gamma_0$ is a torsion group.
\end{proof}

\begin{remark}
We remark here that Theorem \ref{singleton} holds without assuming that $\Gamma_0$ is finitely generated and the the action of $\Gamma_0$ on $X$ is properly discontinuous.
\end{remark}

\section{Generating free groups}\label{sec-fg} \hfill

In this section we prove the first part of Theorem \ref{Tits alternative} using a standard Ping-pong argument. More precisely,

\begin{proposition}
\label{free subgroup}
Let $X$ be a proper visibility space, and let $\Gamma$ be a group acting properly discontinuously by isometries on $X$ such that the limit set satisfies
\[|\cL(\Gamma)|\geq 3.\]
Then $\Gamma$ contains a free subgroup of rank $2$.
\end{proposition}

We split the proof into several steps.
\begin{lemma}\label{fg-p}
Let $X, \Gamma$ be as in Proposition \ref{free subgroup}. If $\Gamma$ contains a parabolic isometry $\varphi$ with a unique fixed point $\eta \in X(\infty)$, then there exists $\phi \in \Gamma$ such that
\[\phi(\eta)\neq \eta.\]
\end{lemma}
\begin{proof}
Suppose for contradiction that $\phi(\eta)=\eta$ for all $\phi\in \Gamma$. By Lemma \ref{stability} we know that $\Gamma$ can not contain any axial isometry. Therefore, all non-trivial element in $\Gamma$ is either parabolic or elliptic. It then follows from \cite[Lemma 2.5]{JiWu} that \[\cL(\Gamma)=\{\eta\},\]
contradicting our assumption that $|\cL(\Gamma)|\geq 3$.
\end{proof}

\begin{lemma}\label{fg-h}
Let $X, \Gamma$ be as in Proposition \ref{free subgroup}. If  $\Gamma$ contains an axial isometry $\varphi$ with two fixed points $\eta^\pm \in X(\infty)$, then there exists $\phi \in \Gamma$ such that
\[\phi(\{\eta^\pm\}) \cap  \{\eta^\pm\}  = \emptyset.\]
\end{lemma}
\begin{proof}
Suppose for contradiction that for all $\phi \in \Gamma$, $$\phi(\{\eta^\pm\}) \cap  \{\eta^\pm\}  \neq \emptyset.$$ 

There are two cases.

\emph{Case (1).} For any $\phi \in \Gamma$, $\phi(\eta^-)=\eta^-$. By Lemma \ref{stability} it follows that for all $\phi \in \Gamma$, $$\phi(\eta^+)=\eta^+$$ as well, and $\Gamma$ is almost infinite cyclic.  In particular, $\varphi$ generates an infinite cyclic subgroup of finite index in $\Gamma$, and $\eta^+$ and $\eta^-$ are the only limit points of $\langle \varphi \rangle$. Since $\langle \varphi \rangle$ has finite index in $\Gamma$, the limit set $\mathcal{L}(\Gamma)$ is also finite. Moreover, as $\mathcal{L}(\Gamma)$ is invariant under $\varphi$,  there exists $k \in \mathbb{N}$ such that $\varphi^k(\tau)=\tau$ for every $\tau \in \mathcal{L}(\Gamma)$, which  implies that every $\tau \in \mathcal{L}(\Gamma)$ is a fixed point of $\varphi$. Therefore, $\mathcal{L}(\Gamma)=\{\eta^\pm\}$, contradicting our assumption that $\cL(\Gamma)\geq 3$.

\emph{Case (2).} There exists some $\phi_0 \in \Gamma$ such that $\phi_0(\eta^-)=\eta^+$.  By Lemma \ref{stability}, for every $\phi \in \Gamma$, one of the following occurs

\begin{enumerate}
     \item $\phi(\eta^-)=\eta^-$, in which case $\phi \in \Gamma_{\eta^-}$;
    \item $\phi(\eta^-)=\eta^+$, in which case $\phi \in \phi_0 \Gamma_{\eta^-}$;
    \item  $\phi(\eta^+)=\eta^-$, in which case $\phi \in \phi^{-1}_0 \Gamma_{\eta^+}=\phi^{-1}_0 \Gamma_{\eta^-}$.
\end{enumerate}
Thus the almost infinite cyclic group $\Gamma_{\eta^-}$ has finite index in $\Gamma$, which implies that $\Gamma$ itself is also almost infinite cyclic. By repeating the argument from \emph{Case (1)}, we may conclude that $|\cL(\Gamma)|=2$, again contradicting our assumption that $|\cL(\Gamma)|\geq 3$.
\end{proof}

Now we are ready to prove Proposition \ref{free subgroup}.

\begin{proof}[Proof of Proposition \ref{free subgroup}]
By Theorem \ref{singleton} we know that the group $\Gamma$ cannot only consist of torsion elements; otherwise $|\cL(\Gamma)|=1$, that contradicts our assumption. Therefore, $\Gamma$ must contain either an axial or a parabolic element $\varphi$. Now we prove it separately. 

If $\varphi$ is parabolic with a unique fixed point $\eta \in X(\infty)$. By Lemma \ref{fg-p}, there exists $\phi \in \Gamma$ such that $\phi(\eta) \neq \eta$. Set
$$\psi=\phi \varphi \phi^{-1}.$$ 
Then we have $\mathrm{Fix}_\psi=\{\phi(\eta)\}\neq \{\eta\}$. Now choose two disjoint open sets  $U,V \subset \overline{X}$ so that $\eta \in U$  and $\phi(\eta) \in V$. By Lemma \ref{parabolic}, there exists a sufficiently large integer $N \in \mathbb{N}^*$, such that for all $n \geq N$ or $n \leq -N$,
\begin{equation*}
\varphi^n (\overline{X} \setminus U) \subset U \textrm{ and } \psi^n (\overline{X} \setminus V) \subset V,
\end{equation*}
in particular,
\begin{equation*}
\varphi^N (V) \subset U \textrm{ and } \psi^N (U) \subset V.
\end{equation*}
By applying the Ping-pong Lemma, i.e. Lemma \ref{Ping-pong Lemma}, we conclude that $\varphi^N$ and $\psi^N$ generate a free subgroup of rank $2$ in $\Gamma$.

If $\varphi$ is axial with two fixed points $\eta^\pm \in X(\infty)$, by Lemma \ref{fg-h} there exists $\phi \in \Gamma$ such that
\[\phi(\{\eta^\pm\}) \cap  \{\eta^\pm\}  = \emptyset.\]
Similarly, set
$$\psi=\phi \varphi \phi^{-1}.$$ 
Then we have $\mathrm{Fix}_\psi=\{\phi(\eta^{\pm})\}$. Now choose two disjoint open sets  $U,V \subset \overline{X}$ so that $\eta^\pm \in U$  and $\phi(\eta^{\pm}) \in V$. By Lemma \ref{axial}, there exists a sufficiently large integer $N \in \mathbb{N}^*$, such that for all $n \geq N$ or $n \leq -N$,
\begin{equation*}
\varphi^n (\overline{X} \setminus U) \subset U \textrm{ and } \psi^n (\overline{X} \setminus V) \subset V,
\end{equation*}
in particular,
\begin{equation*}
\varphi^N (V) \subset U \textrm{ and } \psi^N (U) \subset V.
\end{equation*}
By applying Lemma \ref{Ping-pong Lemma}, we conclude that $\varphi^N$ and $\psi^N$ generate a free subgroup of rank $2$ in $\Gamma$.

The proof is complete.
\end{proof}

\begin{remark}
We are very grateful to Anders Karlsson for pointing out that a nonabelian free subgroup can be constructed under more general conditions than visibility. Let $X$ be a proper CAT($0$) space and $x_0 \in X$ a basepoint. For $W \subset X$ and $C \in \mathbb{R}$, one may define the \emph{halfspace}
\begin{equation*}
H(W,C):=\{z: d(z,W) \leq d(z,x_0)+C\},
\end{equation*}
and the \emph{star} of a boundary point $\eta \in X(\infty)$ as
\begin{equation*}
S(\eta):=\overline{\bigcup \limits_{C\geq 0} \bigcap \limits_{V \in V_\eta}\overline{H(V,C)},}
\end{equation*}
where $V_\eta$ is the collection of open neighborhoods of $\eta$ in the cone topology. A boundary point $\eta \in X(\infty)$ is said to be \emph{hyperbolic} if 
\begin{equation}
\label{hyperbolic point}
S(\eta)=\{\eta\}. 
\end{equation}
It was shown by Karlsson in  \cite[Theorem 31]{DYIS} that \emph{if a group $\Gamma$ acts properly discontinuously  on a proper CAT($0$) space with $|\mathcal{L}(\Gamma)| \geq 3$, and suppose that at least one point in $\mathcal{L}(\Gamma)$ is hyperbolic in the sense of \eqref{hyperbolic point}, then there exist two axial isometries $\varphi,\psi \in \Gamma$ whose disjoint limit sets consist only of hyperbolic points, and which generate a nonabelian free subgroup.} A key ingredient in the proof of \cite[Theorem 31]{DYIS} is a contraction lemma formulated in terms of halfplanes and stars.  In the visiblity setting, it is not hard to see that every boundary point is hyperbolic. Therefore, Proposition \ref{free subgroup} can also follows from \cite[Theorem 31]{DYIS}. We still include our proof here because it is quite self-contained; moreover, we have certain by-products like Theorem \ref{singleton}, that is of independent interest and is useful to study Question \ref{torsion group}.  
\end{remark}

\begin{remark}
\label{infinitely generated}
The conclusion of the first part of Theorem \ref{Tits alternative} holds without assuming that $\Gamma$ is finitely generated or  that $X$ satisfies the bounded packing property.
\end{remark}

The proof of Proposition \ref{free subgroup} yields the following generalization of Lemma \ref{stability}.
\begin{corollary}
\label{axial cyclic}
Let $\Gamma$ be a group acting properly discontinuously by isometries on a proper visibility space $X$. If $\Gamma$ contains an axial element $\varphi$ and satisfies $|\cL(\Gamma)|=2$, then $\Gamma$ is almost infinite cyclic.
\end{corollary}
\begin{proof}
Let $\eta,\xi$ be the fixed points of $\varphi$. Since $|\cL(\Gamma)|=2$, every $\phi \in \Gamma$ preserves $\{\eta,\xi\}$ setwise. If $\phi(\eta)=\eta$ for all $\phi \in \Gamma$, then by Lemma \ref{stability},  $\Gamma=\Gamma_\eta$ is almost infinite cyclic. Otherwise, there exists $\phi_0 \in \Gamma$  such that  $\phi_0(\eta)=\xi$.  In this case, for every $\phi \in \Gamma$, 
\begin{itemize}
\item either $\phi(\eta)=\eta$, hence $\phi \in \Gamma_\eta$;
\item or $\phi(\eta)=\xi$, hence $\phi_0^{-1} \circ \phi \in \Gamma_\eta$.
\end{itemize}
It follows that $\Gamma \subset \Gamma_\eta \cup \varphi_0 \Gamma_\eta$, and hence the almost infinite cyclic group $\Gamma_\eta$ has finite index in $\Gamma$. Therefore, $\Gamma$ is also almost infinite cyclic.
\end{proof}

\section{Almost nilpotency} \hfill

Let $X$ be a proper visibility space. In this section, we  further assume that $X$ satisfies the bounded packing property. Our aim is to prove Theorem \ref{Tits alternative}. By Proposition \ref{free subgroup}, it remains to prove the second part of Theorem \ref{Tits alternative}.
\begin{theorem}\label{fpnil}
Let $X$ be a proper visibility space satisfying the bounded packing property, and let $\Gamma_0$ be a finitely generated group acting properly discontinuously by isometries on $X$. If $|\mathcal{L}(\Gamma)| \leq 2$, then $\Gamma_0$ is almost nilpotent.
\end{theorem}

To derive almost nilpotency in this case, we need the following remarkable theorem of Breuillard-Green-Tao, which is a generalized Margulis Lemma for metric spaces originally conjectured by Gromov.
\begin{theorem}(\cite[Corollay 11.17]{MR3090256})
\label{generalized Margulis}
Let $X$ be a metric space satisfying the bounded packing property, and let $\Gamma$ be a group of isometries acting properly discontinuously on $X$. Suppose that every geodesic ball of radius $4$ in $X$ can be covered by at most $P$ balls of radius $1$. Then there exists a constant $\varepsilon=\varepsilon(P)>0$ such that for every $x \in X$, the group $$\Gamma_\varepsilon(x)= \langle \varphi \in \Gamma: d_\varphi(x) \leq \varepsilon \rangle$$ is almost nilpotent.
\end{theorem}

Following the strategy outlined in \cite{JiWu}, we will first show that the group $\Gamma_0$ in Theorem \ref{fpnil} is amenable. We now recall the following important property which is inspired by  the work of Karlsson-Margulis \cite{MR1729880}
\begin{proposition}(\cite[Proposition 5.3]{JiWu})\label{sublinear}
Let $X$ be a proper visibility CAT(0) space, and let $\Gamma_0$ be a finitely generated group of isometries of $X$. Assume that there exists a point $\eta \in X(\infty)$ such that $\varphi(\eta)=\eta$ for every $\varphi \in \Gamma_0$, and  that $\varphi$ preserves every horosphere centered at $\eta$. Then for any $x\in X$ we have
\[\lim\limits_{k\to \infty}\frac{\max_{|\varphi|_w\leq k}d(x,\varphi(x))}{k}=0,\]
where $|\cdot|_w$ denotes the word norm of $\Gamma_0$.
\end{proposition}
\noindent We remark that \cite[Proposition 5.3]{JiWu} originally requires that for every $\varphi \in \Gamma_0$, the fixed set $\mathrm{Fix}_\varphi$ is exactly $\{\eta\}$. However, upon a careful examination of the proof, the only place that this assumptions is applied is in the line above \cite[Equation (26)]{JiWu}, which clearly holds under the weaker assumption that $\varphi$ preserves every horosphere centered at $\eta$ for every $\varphi \in \Gamma_0$. Therefore, the proof of \cite[Proposition 5.3]{JiWu} also yields the preceding Proposition \ref{sublinear}.\\

To prove Theorem \ref{fpnil}, if $\Gamma_0$ contains an axial element, it is known from Corollary \ref{axial cyclic} that $\Gamma_0$ is almost infinite cyclic, and in particular almost nilpotent. So it remains to consider the case when $\Gamma_0$ contains no axial elements.

\begin{lemma}
\label{horosphere-l}
Let $X, \Gamma_0$ be as in Theorem \ref{fpnil}. Suppose that every element of $\Gamma_0 \setminus \{\mathrm{Id}\}$ is either elliptic or parabolic. Then $\Gamma_0$ fixes a point $\eta \in X(\infty)$ and preserves every horosphere centered at $\eta$ setwise; moreover, $\cL(\Gamma_0)=\{\eta\}.$
\end{lemma}
\begin{proof}
If $\Gamma_0$ consists entirely of elliptic elements, then by Theorem \ref{singleton}, $\Gamma_0$ has a global fixed point in $X(\infty)$. 

If $\Gamma_0$ contains a parabolic element $\varphi$ with the fixed point $\eta \in X(\infty)$, then  clearly $\eta$ is a limit point of $\Gamma_0$. For any $\phi \in \Gamma_0$, the conjugate $\psi=\phi \varphi \phi^{-1}$ is also parabolic and fixes $\phi(\eta)$. If $\phi(\eta) \neq \eta$, then it is easy to see that $\eta, \phi(\eta)$ and $\psi(\eta)$ are three distinct points in $\mathcal{L}(\Gamma)$, which contradicts the assumption that $|\cL(\Gamma)|\leq 2$. Therefore $\eta$ is fixed by every element of $\Gamma$.

To show that $\Gamma$ preserves every horosphere centered at $\eta$ and $\cL(\Gamma_0)=\{\eta\}$, observe that for any elliptic element, the translation length is zero by definition; for a parabolic element $\varphi$, it is known from \cite{Wu18} that $|\varphi|=0$. The assertion then follows from \cite[Lemma 2.5]{JiWu}. 
\end{proof}

The following result generalizes  \cite[Proposition 1.5]{JiWu} to visibility spaces. In this setting, the bounded packing condition serves a role analogous to the Bishop-Gromov volume comparison theorem in Riemmanian geometry.

\begin{lemma}
\label{Main 1}
Let $X$ be a visibility CAT($0$) space satisfying the bounded packing property. And let $\Gamma_0$ be a finitely generated group acting properly discontinuously by isometries on $X$. If $\Gamma_0$ fixes a boundary point $\xi \in X(\infty)$, then $\Gamma_0$ is amenable.
\end{lemma}

\begin{proof} In view of Corollary \ref{axial cyclic}, the interesting case is when $\Gamma_0$ contains no axial element. Let $S$ be a finite generating set for $\Gamma_0$. Combining Proposition \ref{sublinear} and Lemma \ref{horosphere-l}, for any $x \in X$,
\begin{equation}
    \label{short}
    \lim_{k \to \infty}\dfrac{\max_{|\varphi |_w \leq k} d(x, \varphi(x))}{k}=0,
\end{equation}
 where $|\cdot|_w$ denotes the word norm with respect to $S$.

Let $$b_k=\max_{|\varphi |_w \leq k} d(x, \varphi(x)).$$

 Since $X$ satisfies the bounded packing property, there exists $P \in \mathbb{N}^*$ such that every ball  of radius $6$ in $X$ can be covered by at most $P$ balls of radius $1$. By Lemma \ref{packing}, a ball of radius $b_k$ can be covered by at most $(1+P)^{b_k}$ balls of radius $1$.

Since the action of $\Gamma_0$ on $X$ is properly discontinuous, there exists $K \in \mathbb{N}^*$ such that for any $y \in X$, $$ \# \{\varphi \in \Gamma_0: \varphi(x) \in B_y(1)\} \leq K.$$ We then obtain
\begin{eqnarray*}
\# \{\varphi \in \Gamma_0: |\varphi|_w \leq k\} &\leq& \# \{\varphi \in \Gamma_0: d(x,\varphi(x))\leq b_k\} \\
& \leq& K (1+P)^{b_k}.
\end{eqnarray*}
Taking logarithms of both sides and applying \eqref{short} to the limit yields
\begin{equation*}
 \limsup_{k \to \infty}\frac{\log (\# \{\varphi \in \Gamma_0: |\varphi|_w \leq k\})}{k}= 0.
\end{equation*}
This shows that $\Gamma_0$ has subexponential growth, and hence is amenable \cite[Page 205]{Harpe-book}.
\end{proof}

To show that $\Gamma_0$ is indeed almost nilpotent, we need to iteratively apply the following results of Adams-Ballmann and Breuillard-Fujiwara in the transverse spaces.

\begin{theorem}(\cite{MR1645958})
\label{AB}
Let $X$ be an unbounded proper CAT($0$) space. If $\Gamma$ is an amenable group acting on $X$ by isometries, then $\Gamma$ fixes either a point in $X(\infty)$ or a flat in $X$ setwise.
\end{theorem}

\begin{proposition}(\cite[Propostion 9.3]{MR4275871})
\label{BF}
Let $S$ be a finite set of isometreis of $\mathbb{R}^n$ and $\Gamma_0= \langle S \rangle$ the group generated by $S$. Denote by $| \cdot |_w$ the word norm with respect to $S$. If
\begin{equation*}
\lim_{k \to \infty}\dfrac{\max_{|\varphi |_w \leq k} d(x, \varphi(x))}{k}=0,
\end{equation*}
Then $\Gamma_0$ has a global fixed point in $\mathbb{R}^n$.
\end{proposition}    

We are now ready to complete the proof of Theorem \ref{fpnil}.
\begin{proof}[Proof of Theorem \ref{fpnil}.]
In view of Lemma Corollary \ref{axial cyclic}, we may assume that $\Gamma_0$ contains no axial elements. Set $\eta_1=\eta$.  By Lemma \ref{packing geometry} and Theorem \ref{Transverse proper}, $X_{\eta_1}$ is a complete CAT($0$) space satisfying the bounded packing property.  By Lemma  \ref{horosphere-l}, $\Gamma_0$ induces a group acting (not necessarily properly discontinuously) by isometries on $X_{\eta_1}$.
 
 If $X_{\eta_1}$ is bounded, by the Cartan fixed point theorem, $\Gamma_0$ fixes the  unique center  of $X_{\eta_1}$. This together with the definition of transverse spaces implies that $$\inf\limits_{x\in X}\max\limits_{1\leq j\leq k}d(\phi_j\circ x, x)=0$$ where $\Gamma_0=\langle \phi_1,\cdots, \phi_k\rangle$. In particular, there exists a point $q_0 \in X$ so that $$d_{\phi_j}(q_0)< \varepsilon \text{ for all } 1\leq j \leq k,$$ where $\varepsilon$ is the Margulis constant as in Theorem \ref{generalized Margulis}. It then follows from Theorem \ref{generalized Margulis} of Breuillard-Green-Tao that $\Gamma_0$ is almost nilpotent.
    
If $X_{\eta_1}$ is unbounded, then by Theorem \ref{AB} we have,
\begin{enumerate}
\item either $\Gamma_0$  preserves a flat $F \subset X_{\eta_1}$ setwise;
\item or $\Gamma_0$ fixes a point $\eta_2 \in X_{\eta_1}(\infty)$. 
\end{enumerate} 
    
    \emph{Case (1).} Observe that for any $x,y \in X$, $$d_{\eta_1}(\pi_{\eta_1}(x),\pi_{\eta_2}(y)) \leq d(x,y).$$ It follows that for any $x \in F$,
    \begin{equation*}
    \lim_{k \to \infty}\dfrac{\max_{|\varphi |_w \leq k} d_{\eta_1}(x, \varphi(x))}{k}=0.
    \end{equation*}
Applying Proposition \ref{BF}, we obtain a point $y \in F$ fixed by $\Gamma_0$. Then similar as the case when $X_{\eta_1} $ is bounded, one may also conclude that $\Gamma_0$ is almost nilpotent.

\emph{Case (2).} Consider the induced isometric action of $\Gamma_0$ on the iterated transverse space $X_{\eta_1,\eta_2}$. Again, by Theorem \ref{AB}, either $\Gamma_0$ fixes a flat in $X_{\eta_1,\eta_2}$, in which case $\Gamma_0$ has a global fixed point in $X_{\eta_1,\eta_2}$, or $\Gamma_0$ fixes another point $\eta_3 \in X_{\eta_1, \eta_2}(\infty)$.  we can iterate this process provided that the  resulting transverse space is unbounded. By Theorem \ref{finite depth}, the iteration terminates after finitely many steps. At the final step, either $\Gamma_0$ fixes a flat in $X_{\eta_1,\ldots,\eta_{m}}$ with $m$ bounded above by the depth of $X$, in which case a global fixed point for $\Gamma_0$ exists; or $X_{\eta_1,\ldots,\eta_{m}}$ is eventually bounded, in which case the existence of a global fixed point for $\Gamma$ is guaranteed by the Cartan fixed point theorem. By continuity, we can select a sequence of points $x_{m-1},\cdots, x_1,x_0$ with $x_i \in  X_{\eta_1,\ldots,\eta_i}$ for $1 \leq i \leq m-1$ and $x_0 \in X$, such that $$d_{\phi_j}(x_i)< \dfrac{m-i}{m} \varepsilon \textrm{ for all } 1\leq j \leq k \textrm{ and  every } 0 \leq i \leq m-1.$$ In particular, we also have $$d_{\phi_j}(x_0)< \varepsilon \text{ for all } 1\leq j \leq k,$$ The almost nilpotency of $\Gamma_0$ then follows from Theorem \ref{generalized Margulis}, i.e. the generalized Margulis Lemma. 

The proof is complete.
\end{proof}

We are now able to provide a complete characterization of the group in Theorem \ref{Tits alternative} in terms of the cardinality of their limit sets.
\begin{theorem}
\label{elementary char}
Let $X, \Gamma_0=\Gamma$ be as in Theorem \ref{fpnil}. Then exactly one of the following holds:
\begin{itemize}
    \item $|\mathcal{L}(\Gamma_0)|=0$ if and only if $\Gamma_0$ is finite;
    \item $|\mathcal{L}(\Gamma_0)|=1$ if and only if every non-identity element of $\Gamma_0$ is either elliptic or parabolic, in which case $\Gamma_0$ is almost nilpotent;
    \item $|\mathcal{L}(\Gamma_0)|=2$ if and only if $\Gamma_0$ contains an axial element, in which case $\Gamma_0$ is almost infinite cyclic.
\end{itemize}
\end{theorem}
\begin{proof}
The first assertion follows immediately from Lemma \ref{stability finite}. Observe that if $\Gamma_0$ contains an axial element, then $|\mathcal{L}(\Gamma_0)| \geq 2$. The second assertion then follows from Lemma \ref{horosphere-l} and Theorem \ref{fpnil}.

For the third assertion, we first show that $|\mathcal{L}(\Gamma_0)|=2$ implies $\Gamma_0$ contains an axial element. Assume the contrary that $\Gamma_0$ contains no axial elements, then by Lemma \ref{fpnil}, $\Gamma_0$ fixes a point $\eta \in X(\infty)$. If $\Gamma_0$ is a torsion group, by Theorem \ref{singleton}, $\eta$ is the unique limit point of $\Gamma_0$, contradicting $|\mathcal{L}(\Gamma_0)|=2$. If $\Gamma_0$ contains a parabolic element $\varphi$, then $\varphi$ preserves $\mathcal{L}(\Gamma_0)$ setwise. Since $\mathcal{L}(\Gamma_0)$ contains the unique fixed point of $\varphi$, the other point in $\mathcal{L}(\Gamma_0)$ must also be fixed by $\varphi$, contradicting the fact that a parabolic isometry of a visibility space has exactly one fixed point. Therefore $\Gamma_0$ must contain an axial element. By Corollary \ref{fpnil}, $\Gamma_0$ is almost infinite cyclic.

It remains to prove that if $\Gamma_0$ contains an axial element $\varphi$, then $|\mathcal{L}(\Gamma_0)|=2$. This follows easily from the observation that both fixed points of $\varphi$ are contained in $\mathcal{L}(\Gamma_0)$, and no other limit points can occur by the assumption that $|\cL(\Gamma_0)|\leq 2$.

The proof is complete.
\end{proof}

Now we finish the proofs of Theorem \ref{Tits alternative} and Theorem \ref{global fixed point}.

\begin{proof}[Proof of Theorem \ref{Tits alternative}]
By Proposition \ref{free subgroup}, if $|\cL(\Gamma)|\geq 3$, then it contains a free subgroup of rank $2$. This completes the proof of Part $(1).$

If $|\cL(\Gamma)|\leq 2$, then by the characterization in Theorem \ref{elementary char}, $\Gamma$ is almost nilpotent in all cases.

The proof is complete.
\end{proof}

\begin{proof}[Proof of Theorem \ref{global fixed point}]
Recall that a finitely generated almost nilpotent torsion group is always finite (see e.g. \cite{MR2825167}). Then the conclusion clearly follows from  Theorem \ref{Tits alternative}.
\end{proof}

\bibliographystyle{amsalpha}
\bibliography{Tits} 
\end{document}